\documentclass[11pt]{article}
\usepackage[T1]{fontenc}
\usepackage[utf8]{inputenc}

\usepackage[margin=1in]{geometry}
\usepackage{amsfonts,amsmath,amssymb,amsthm}
\usepackage{cite,microtype,graphicx,hyperref}
\usepackage[capitalize,nameinlink]{cleveref}

\newtheorem{theorem}{Theorem}
\newtheorem{lemma}{Lemma}

\theoremstyle{definition}
\newtheorem{definition}{Definition}

\newcommand{\subdub}[1]{\operatorname{SDD}(#1)}

\author{David Eppstein\thanks{Department of Computer Science, University of California, Irvine. This research was supported in part by the National Science Foundation under grant CCF-2212129.}}
\title{Hamiltonian Cycles in Subdivided Doubles}

\date{ }

\begin{document}
\maketitle  


\begin{abstract}
The subdivided double construction on 4-regular graphs was used by Potočnik and Wilson to explore semi-symmetric (edge-transitive but not vertex-transitive) graphs, and can be used to construct every semi-symmetric 4-regular graph that contains a pair of twin vertices. We show that (regardless of symmetry) subdivided doubles have another curious property: they have exponentially many Hamiltonian cycles each of which is complementary to another Hamiltonian cycle.
\end{abstract}


\section{Introduction}
In a 4-regular graph, the complement of a Hamiltonian cycle is another 2-regular spanning subgraph, which might or might not be another Hamiltonian cycle. In $K_5$, for instance, the only 2-regular spanning subgraphs are 5-cycles, so each Hamiltonian cycle is complementary to another one. In the octahedral graph $K_{2,2,2}$, however,  some Hamiltonian cycles are complementary to pairs of disjoint triangles. Because there is no reason to expect any special structure such as connectedness in the complements of Hamiltonian cycles, we expect the behavior of $K_{2,2,2}$ to be the more typical case for large Hamiltonian 4-regular graphs. Potočnik and Wilson used \emph{subdivided doubles}, a method for transforming 4-regular graphs into larger 4-regular graphs, to analyze the structure of semi-symmetric graphs (the regular and edge-transitive but not vertex-transitive graphs)~\cite{PotWil-JCTB-07,PotWil-AMC-14}. This construction produces graphs having many Hamiltonian cycles (some of them coming from Euler tours of the underlying graph). In this paper we prove that all subdivided doubles share with $K_5$ the property that every Hamiltonian cycle is complementary to another Hamiltonian cycle. For our convenience, we give this property a name:

\begin{definition}
A 4-regular graph is \emph{Hamiltonian-paired} if it is Hamiltonian and the complement of every Hamiltonian cycle is another Hamiltonian cycle.
\end{definition}

With this definition established, we can express our main result succinctly: every subdivided double is Hamiltonian-paired.

When a 4-regular graph has two complementary Hamiltonian cycles, they form a \emph{Hamiltonian decomposition}.
Every 4-regular graph has an even number of Hamiltonian decompositions~\cite{Tho-ADM-78}.
A line graph of a cubic graph has a Hamiltonian decomposition if and only if the underlying cubic graph has a Hamiltonian cycle~\cite{Kot-CPPM-57,Mar-AM-76}. Because testing Hamiltonicity of cubic planar graphs is $\mathsf{NP}$-complete~\cite{GarJohSto-TCS-76}, it is also $\mathsf{NP}$-complete to test the existence of a Hamiltonian decomposition in the line graph of a cubic planar graph, and more generally in any 4-regular planar graph. These remarks also resolve a 2012 conjecture of Cygan et al.~\cite{CygHouKow-JGT-12} on the complexity of linear arboricity. A \emph{linear forest} is a union of vertex-disjoint paths, and the \emph{linear arboricity} of a graph is the minimum number of linear forests into which the edges of the graph can be partitioned. If a 4-regular graph has a Hamiltonian decomposition, then removing an arbitrary vertex turns its two Hamiltonian cycles into two Hamiltonian paths, showing that the graph with the removed vertex has linear arboricity two. On the other hand, if a 4-regular graph does not have a Hamiltonian decomposition, then removing a single vertex cannot produce a graph with linear arboricity two, because two linear forests that are not both Hamiltonian paths do not have enough edges to cover all edges of the resulting graph, and a cover by two Hamiltonian paths would necessarily terminate the paths at the four odd-degree vertices left as the neighbors of the removed vertex, giving a Hamiltonian decomposition. Because it is $\mathsf{NP}$-complete to test the existence of a Hamiltonian decomposition in a 4-regular planar graph, it is $\mathsf{NP}$-complete to test whether the subgraph obtained by removing any single vertex has linear arboricity two. Therefore, more generally, it is $\mathsf{NP}$-complete to test whether a planar graph has linear arboricity two, as was conjectured by Cygan et al. Almost all 4-regular graphs, and almost all 4-regular bipartite graphs, have Hamiltonian decompositions~\cite{FriJerMol-JAlg-96,KimWor-JCTB-01,GreKimWor-JCTB-04} Our results show that, despite the hardness of finding Hamiltonian decompositions in arbitrary 4-regular graphs, they are easy to find in subdivided doubles.

\begin{figure}[b!]
\centering\hrule\medskip
\includegraphics[width=\textwidth]{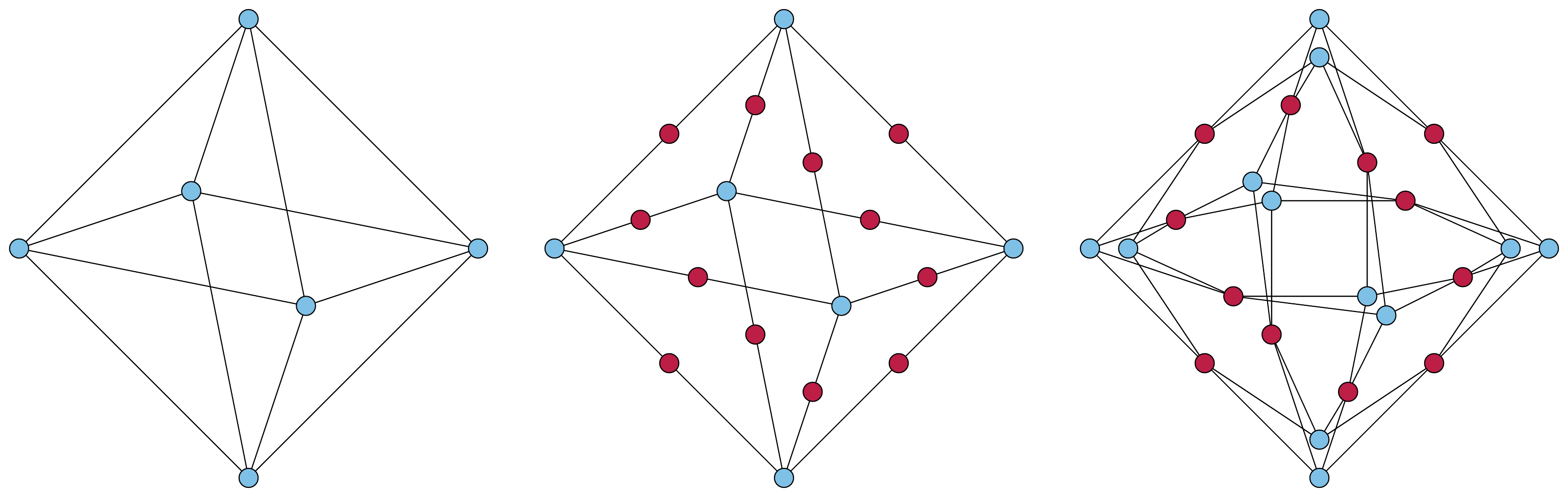}
\caption{The two steps of the subdivided double construction, applied to the octahedral graph $K_{2,2,2}$ (left): subdivide each edge (center), and then double each degree-four vertex of the result to produce the subdivided double $\subdub{K_{2,2,2}}$ (right).}
\label{fig:subdub-octahedron}
\end{figure}

\section{Subdivided doubles}

We begin by reviewing the definition of a subdivided double from Potočnik and Wilson~\cite{PotWil-JCTB-07,PotWil-AMC-14}. Following the notation of Potočnik, Verret, and Wilson~\cite{PotVerWil-DAM-21}, we denote by $\subdub{G}$ the subdivided double of any 4-regular graph $G$. 

\begin{definition}
\label{def:subdub}
$\subdub{G}$ is constructed as follows. First, subdivide each edge of $G$ into a two-edge path, producing a bipartite biregular graph with degree-four vertices on one side of the bipartition (the original vertices of $G$) and degree-two vertices on the other side (the subdivision vertices). Then, replace each degree-four vertex by a pair of twin vertices, not adjacent to each other but sharing the same four neighbors. The resulting graph is 4-regular, as the replacement of each degree-four vertex by twins doubles the degree of the subdivision vertices from two to four, while not changing the degree of the replaced twin vertices.
\end{definition}

\cref{fig:subdub-octahedron} shows the application of this process to the octahedral graph $K_{2,2,2}$, producing its subdivided double $\subdub{K_{2,2,2}}$.

\begin{figure}[b!]
\centering\hrule\medskip
\includegraphics[scale=0.35]{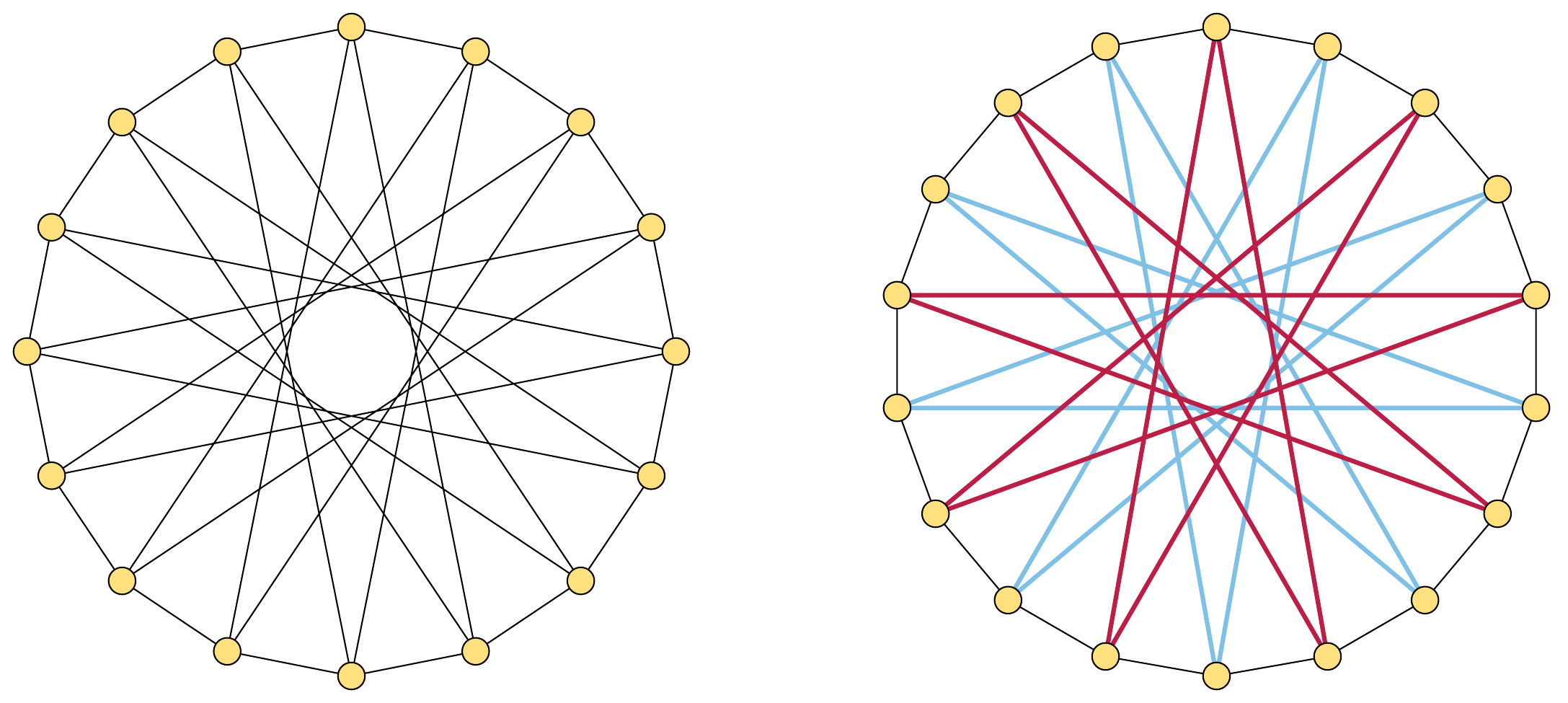}
\caption{The circulant graph $C_{16}^{1,7}$ formed as the subdivided double of a 4-cycle with doubled edges (left) is Hamiltonian-paired; the outer polygon and inner star polygon form a Hamiltonian decomposition. The circulant graph $C_{18}^{1,8}$ (right) is not Hamiltonian-paired, despite having the same local structure; the outer polygon is a Hamiltonian cycle but its complement is a pair of cycles (red and blue) rather than a single Hamiltonian cycle.}
\label{fig:circulant}
\end{figure}

As Potočnik and Wilson observe, the subdivided double of any arc-transitive simple graph is edge-transitive, and every 4-regular semi-symmetric graph with a pair of twin vertices arises in this way. In particular, the Folkman graph (\cref{fig:euler2ham,fig:folkman}), the smallest semi-symmetric graph, is the subdivided double $\subdub{K_5}$ of $K_5$~\cite{PotWil-JCTB-07,PotWil-AMC-14}. It is also possible to apply the subdivided double construction to 4-regular multigraphs as well as to simple graphs. In particular:
\begin{itemize}
\item If $D_n$ denotes the two-vertex \emph{dipole multigraph} with $n$ edges connecting its two vertices, then $\subdub{D_4}=K_{4,4}$.
\item If $\bar{\bar C}_n$ denotes the multigraph formed by doubling each edge in an $n$-vertex cycle, then $\subdub{\bar{\bar C}_n}$ is a circulant graph $C_{4n}^{1,2n-1}$, formed from a cyclic sequence of $4n$ vertices by connecting pairs of vertices that are one or $2n-1$ step apart in the sequence. The twin of each vertex is the vertex opposite it in the cyclic sequence. \cref{fig:circulant} depicts the circulant $C_{16}^{1,7}$ arising from the case $n=4$. (These graphs can also be constructed as the lexicographic product $C_{2n}\cdot 2K_1$ of a $2n$-vertex cycle graph and a two-vertex edgeless graph.)
\end{itemize}
Our results on subdivided doubles apply to the subdivided doubles of multigraphs, and not only to the subdivided doubles of simple graphs.

It is easy to recognize the subdivided doubles algorithmically. They are the 4-regular bipartite graphs in which the vertices of one side of the bipartition can be paired up into twins. When this is the case, the underlying graph of which it is the subdivided double can be obtained by condensing each pair of twins on that side of the bipartition into a single vertex, which causes all of their neighbors to have degree two, and then contracting the resulting paths through these neighbors into single edges.

\section{Hamiltonicity}
To avoid trivial cases, our definition of Hamiltonian-paired graphs requires these graphs to be Hamiltonian. To this end, we now show that subdivided doubles have many Hamiltonian cycles.

\begin{theorem}
\label{thm:many-ham}
Let $G$ be a connected 4-regular multigraph with $n$ vertices. Then $\subdub{G}$ has at least $2^n$ Hamiltonian cycles.
\end{theorem}

\begin{figure}[b!]
\centering\hrule\medskip
\includegraphics[width=\textwidth]{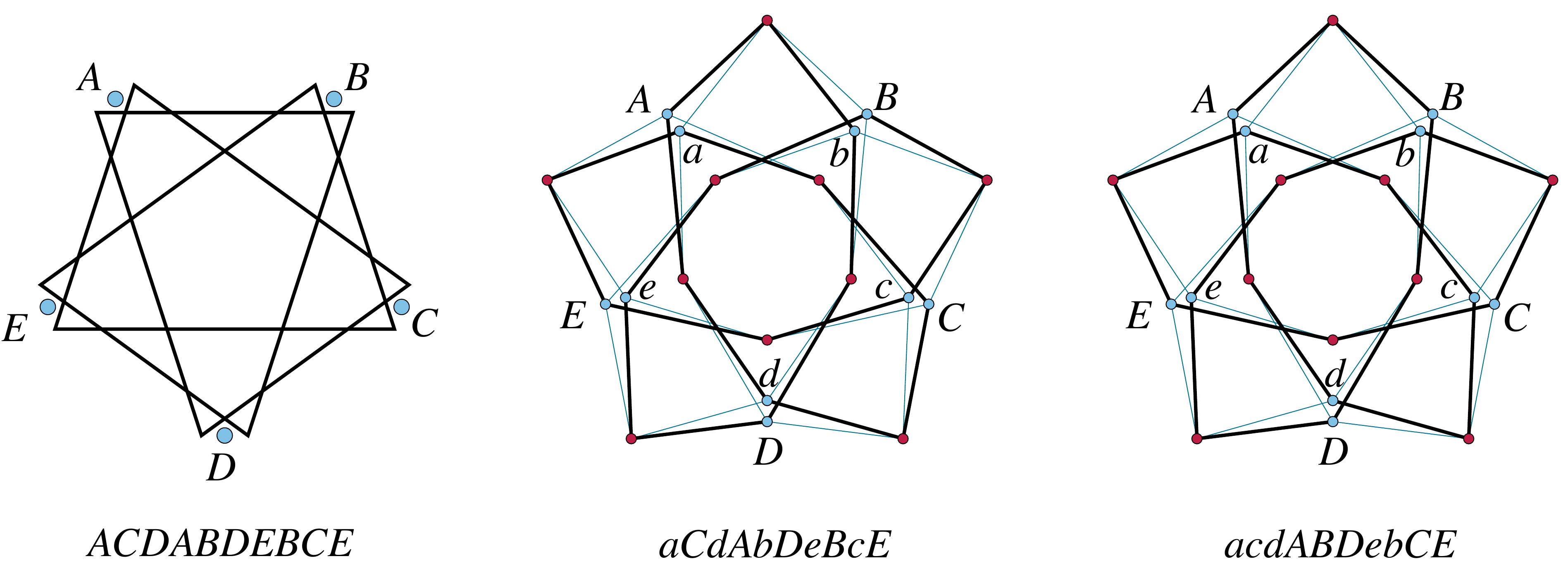}
\caption{Left: An Euler tour in $K_5$. Center and right: the subdivided double of $K_5$ (the Folkman graph),
showing two out of the 32 ways of replacing each occurrence of a vertex in the Euler tour with one of its two copies in the Folkman graph, and the two resulting Hamiltonian cycles in the Folkman graph.}
\label{fig:euler2ham}
\end{figure}

\begin{proof}
We describe an algorithm for constructing a Hamiltonian cycle in $G$, based on certain arbitrary choices, and prove that the stated number of distinct choices lead to distinct cycles. To construct a Hamiltonian cycle, perform the following steps:
\begin{itemize}
\item As is well known, every connected multigraph with even vertex degrees has an Euler tour. Choose $T$ to be any Euler tour of $G$. In a simple graph, this tour can be represented by the cyclic order in which it visits the vertices; for example, \cref{fig:euler2ham} (left) shows an Euler tour $T$ in the complete graph $K_5$ that visits the vertices in the order $ACDABDEBCE$. In a multigraph, we additionally specify which edge of $G$ connects each consecutive pair of vertices in the cyclic order.
\item Because $G$ is 4-regular, $T$ must pass through each vertex of $G$ exactly twice. Choose arbitrarily a bijection between these two visits of each vertex in $T$, and the two twin copies of the same vertex in $\subdub{G}$. This bijection can be represented by replacing each vertex of $G$, in the cyclic order representing the Euler tour, by one of its two copies in $\subdub{G}$. Two choices of this bijection are shown in \cref{fig:euler2ham} (center and right), each labeling the two copies of each vertex in upper or lower case and using the same case convention for the strings representing the cyclic ordering of vertices.
\item Form a walk in $\subdub{G}$ that visits the vertices of the resulting cyclic order in this order, by replacing each edge $e$ in tour $T$ by a two-edge path in $\subdub{G}$ through the subdivision vertex corresponding to $e$. (\cref{fig:euler2ham} illustrates this construction for each of the two chosen bijections.)
\end{itemize}
The walk resulting from this construction visits each subdivision vertex once, because the Euler tour $T$ from which it was constructed visits each edge once. It visits each twin vertex in $\subdub{G}$ once, because each twin in $\subdub{G}$ corresponds to one of the two visits of $T$ to the corresponding vertex in $G$. As a walk that visits each vertex once, it must be a Hamiltonian cycle.

For each vertex in $G$, there are two choices for the bijection between its two visits in $T$ and its two twin copies in $\subdub{G}$. Thus, the second step of the construction, in which we choose this bijection, has a total of $2^n$ choices. For any fixed choice of $T$, the bijection chosen in the second step can be uniquely reconstructed from the resulting Hamiltonian cycle, by examining for each edge of $G$ which of the six possible two-edge paths the tour follows through the subdivision vertex of that edge. Therefore, once $T$ is chosen, each of the $2^n$ choices for the bijection leads to a distinct Hamiltonian cycle.
\end{proof}

In many cases, the number of Hamiltonian cycles will be significantly larger than the bound in \cref{thm:many-ham}, for two reasons. First, $G$ will generally have many Euler tours, and the proof of the theorem uses only one of them. And second, not every Hamiltonian cycle of $\subdub{G}$ may come from an Euler tour in this way. The following definition and lemma examine the structure of Hamiltonian cycles that do not come from Euler tours.

\begin{definition}
Let $H$ be a Hamiltonian cycle in $\subdub{G}$. We define a \emph{hairpin} in $H$  to be a path of two consecutive edges in $H$ whose endpoints were constructed as twins in $\subdub{G}$.
\end{definition}

\begin{lemma}
Let $G$ be a graph or multigraph without self-loops. Then a Hamiltonian cycle in $\subdub{G}$ comes from an Euler tour in $G$, using the construction in the proof of \cref{thm:many-ham}, if and only if it has no hairpin.
\end{lemma}

\begin{proof}
Because $G$ has no self-loops, each two consecutive vertices in an Euler tour of $G$ are distinct, and the construction in \cref{thm:many-ham} will produce a tour in which each two consecutive vertices that are copies of vertices in $G$ are copies of distinct vertices; that is, it will not produce a hairpin.

In the other direction, let $H$ be a Hamiltonian cycle in $\subdub{G}$ that has no hairpin. We will show that $H$ must come from this construction. Because $\subdub{G}$ is bipartite, with its subdivision vertices and twin vertices forming the two sets of its bipartition,
$H$ must alternate between subdivision vertices and twin vertices. Its cyclic subsequence of subdivision vertices corresponds, in $G$, to a cyclic sequence of edges, with each two consecutive edges sharing an endpoint. A hairpin, in $H$, corresponds to a cyclic sequence in which both neighbors of some edge share the same endpoint, which fails to be a walk in $G$. If there are no hairpins, then the cyclic sequence of edges in $G$ forms a walk passing through both endpoints of each edge; that is, it is an Euler tour.
\end{proof}

\begin{figure}[t]
\centering\includegraphics[width=\textwidth]{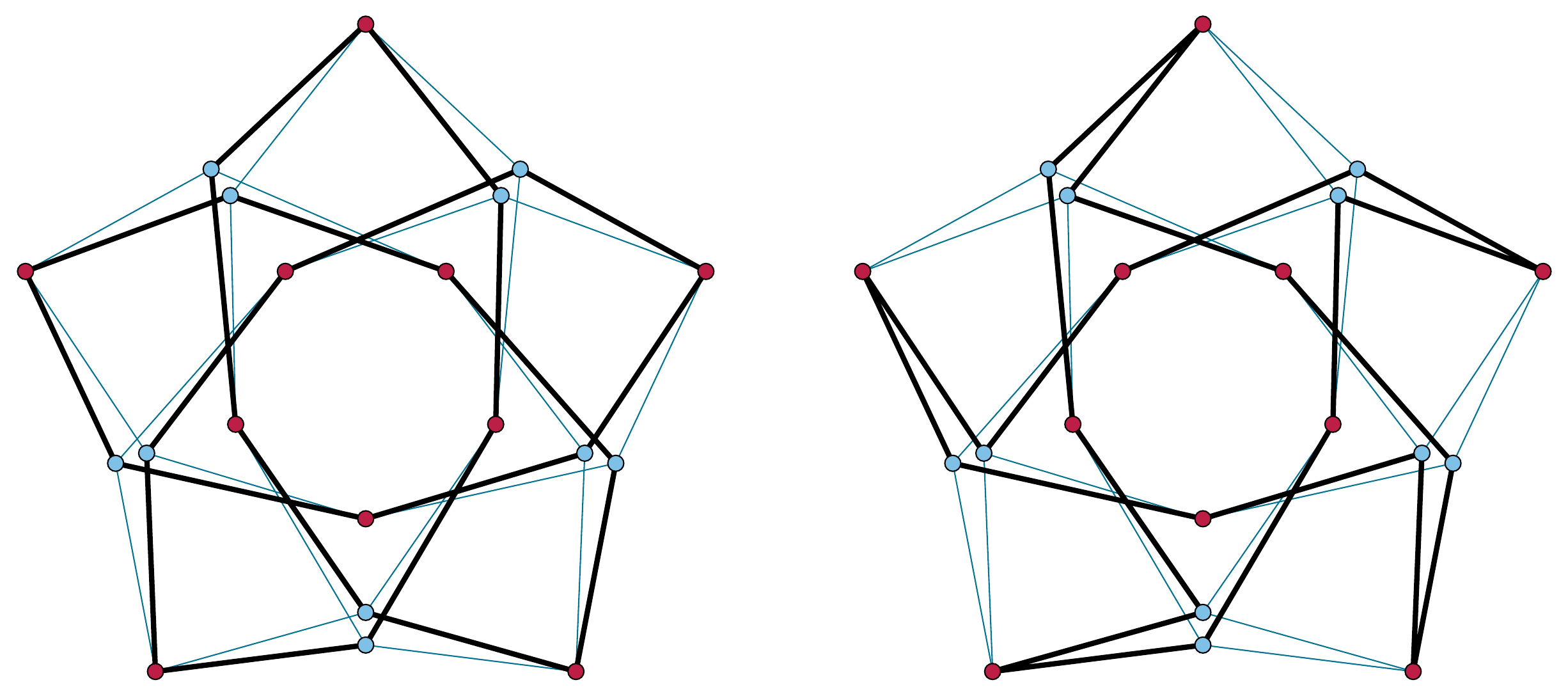}
\caption{The Folkman graph as a subdivided double $\subdub{K_5}$, with two Hamiltonian decompositions: left, without hairpins, and right, with hairpins. Twin vertices are light blue, and are placed close to their twins; subdivision vertices are red. In each case the thick black edges form one Hamiltonian cycle and the complementary thin blue edges form another.}
\label{fig:folkman}
\medskip\hrule
\end{figure}

For instance, in \cref{fig:folkman}, the Hamiltonian cycle drawn as a circle in the left drawing of the Folkman graph has no hairpins, and comes from an Euler tour. The Hamiltonian cycle in the right drawing includes hairpins, and cannot be constructed from an Euler tour.

\section{Hamiltonian decomposition} 
Our main result is the following:

\begin{theorem}
\label{thm:subdub}
Every subdivided double of a connected 4-regular multigraph is Hamiltonian-paired.
\end{theorem}

\begin{proof}
We consider separately the special case $\subdub{B_2}$, where $B_2$  is a bouquet of two loops on a single vertex (considering each loop to add two to the degree of its vertex). Its subdivided double is $\subdub{B_2}=\bar{\bar C}_4$, a four-cycle with each edge doubled. In $\bar{\bar C}_4$, every Hamiltonian cycle uses one copy of each doubled edge. The complement of a Hamiltonian cycle consists of the other copy of each doubled edge, forming another Hamiltonian cycle. Therefore, every Hamiltonian cycle is complementary to another Hamiltonian cycle. In all remaining cases we may assume without loss of generality that the underlying multigraph $G$ has more than one vertex and therefore that its subdivided double $\subdub{G}$ has at least eight vertices: at least four twin vertices and at least four subdivision vertices.

By \cref{thm:many-ham} the subdivided double of a connected 4-regular graph or multigraph is Hamiltonian so it remains to prove that the complement of every Hamiltonian cycle is another Hamiltonian cycle. Therefore, let $\subdub{G}$ be the subdivided double of a given connected 4-regular graph or multigraph $G$, and let $H$ be any Hamiltonian cycle in $\subdub{G}$. For each pair of twin vertices $v$ and $v'$ in $\subdub{G}$, let $a$, $b$, $c$, and $d$ be the four subdivision vertices that neighbor $v$ and $v'$, and consider the two two-edge paths by which $H$ passes through $v$ and $v'$. These two paths may be disjoint, or they may share one subdivision point as a common endpoint. (They cannot share two common endpoints, because that would form a four-vertex cycle and we are assuming that $H$ is a Hamiltonian cycle of a graph with more than four vertices.)
\begin{itemize}
\item If the pair of two-edge paths through twin vertices are disjoint (\cref{fig:ham-nbhd}, left), we may assume (by permuting the labels of the vertices as necessary) that they are $avb$ and $cv'd$. They do not form any hairpins. In this case, the edges in the neighborhood of $v$ and $v'$ that are not part of $H$ again form two disjoint paths, $av'b$ and $cvd$, again without hairpins.
\item If the pair of two-edge paths through twin vertices have a common endpoint, we may assume (by permuting the labels of the vertices as necessary) that they are $avb$ and $bv'c$, with shared endpoint $b$. They connect to form a single four-edge path $avbv'c$, with the middle two edges forming a single hairpin and the outer two edges not part of any hairpin (\cref{fig:ham-nbhd}, right). In this case, the edges in the neighborhood of $v$ and $v'$ that are not part of $H$ again form a four-edge path $av'dvc$ with a single hairpin formed by its middle two edges.
\end{itemize}
\begin{figure}[t]
\centering\includegraphics[scale=0.3]{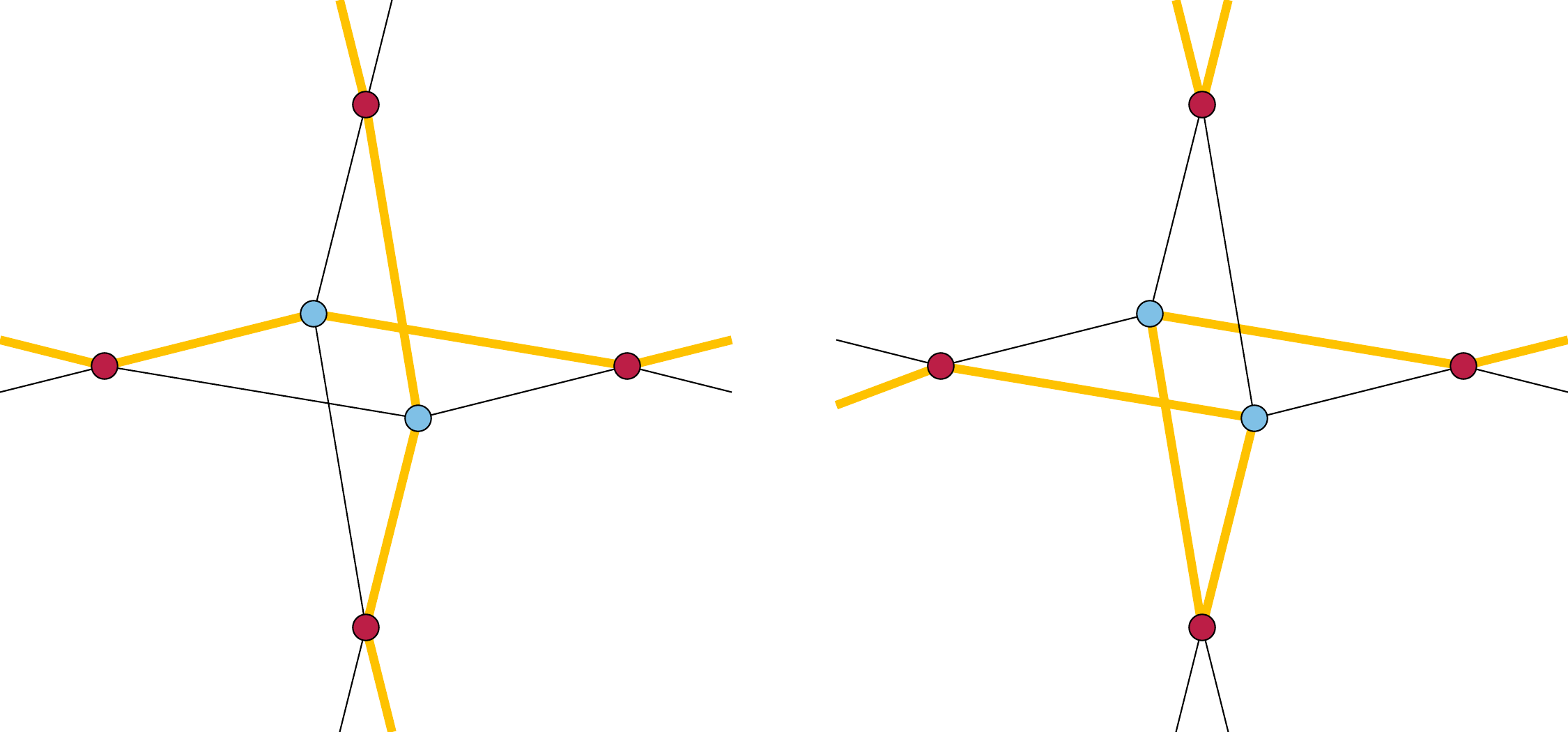}
\caption{Possible neighborhoods of twin vertices in a Hamiltonian cycle of a subdivided double. Left: No hairpins. Right: The Hamiltonian cycle and its complement each have a single hairpin.}
\label{fig:ham-nbhd}
\medskip\hrule
\end{figure}
Thus, in each case, the complement of $H$ has the same local structure as $H$ in the neighborhood of each pair of twin vertices: either $H$ and its complement both have a single hairpin, or they both have no hairpin. The complement of $H$ is automatically 2-regular, because it is the complement of a Hamiltonian cycle in a 4-regular graph. We must prove that it forms a connected 2-regular graph, and so, is a Hamiltonian cycle.

To do so, consider the cyclic sequence of subdivision vertices in $\subdub{G}$ that are not part of hairpins in $H$, and let $C$ be the corresponding cyclic sequence of edges in $G$. In each of the two neighborhoods in \cref{fig:ham-nbhd}, each path in $H$ connects two of these non-hairpin subdivision vertices, so $C$ forms a closed walk in $G$. It visits all vertices  in $G$, either once (for the pairs of twins with a hairpin in their neighborhood) or twice (for the pairs of twins with no hairpin). Applying the same construction to the complement $\bar H$ of $H$ produces a closed walk $C'$ in $G$ in the same way, and connects the same pairs of subdivision vertices to each other, so it produces exactly the same cyclic sequence of edges in $G$: that is, $C$ and $C'$ are identical. Because $\bar H$ maps to a closed walk $C'$ that visits all vertices in $G$, $\bar H$  forms a single cycle that visits all twin pairs in $\subdub{G}$, rather than forming two or more cycles. $\bar H$  must also visit every subdivision vertex in $\subdub{G}$, because each of these vertices has degree four and is touched by two edges of $H$, so its remaining two edges must be in $\bar H$. Because $\bar H$ is a single cycle that visits every vertex, it is a Hamiltonian cycle.
\end{proof}

Although this proof uses heavily the local structure of subdivided doubles, it needs also the global structure of these graphs. For instance, the proof applies to the circulant graphs $C_{4n}^{1,2n-1}$, which are subdivided doubles. However, the circulant graphs $C_{4n+2}^{1,2n}$ have the same local structure (they have the same radius-$r$ neighborhoods of vertices for small $r$) but are not subdivided doubles and are not Hamiltonian-paired. In these graphs, the Hamiltonian cycle formed by the edges connecting consecutive vertices is complementary to a pair of disjoint cycles formed by the edges that connect vertices $2n$ steps apart. The proof of \cref{thm:subdub} uses the global partition of vertices in $\subdub{G}$ into twins and subdivision vertices, and no consistent global partition of this type can be defined for~$C_{4n+2}^{1,2n}$.
 
\bibliographystyle{plainurl}
\bibliography{hampan}

\end{document}